\documentclass[12pt]{amsart}
\usepackage{amssymb,amsmath}
\usepackage{verbatim}
\usepackage{booktabs}
\usepackage{enumerate}
\usepackage{hyperref}
\usepackage{graphicx}
\usepackage{subfigure}

\newtheorem{theorem}{Theorem} \theoremstyle{definition}

\newtheorem{lemma}[theorem]{Lemma}

\theoremstyle{remark}
\newtheorem{remark}{Remark}
\numberwithin{remark}{section}
\numberwithin{theorem}{section}
\numberwithin{equation}{section}
\numberwithin{definition}{section}

\newcommand{\tri}{\mathcal{T}}
\newcommand{\triH}{\tri_H}

\newcommand{\grad}{\nabla}
\newcommand{\gmin}{\alpha}
\newcommand{\gmax}{\beta}
\newcommand{\dx}{\operatorname*{d}\hspace{-0.3ex}x}

\newcommand{\ddiv}{\operatorname*{div}}
\newcommand{\dist}{\operatorname*{dist}}
\newcommand{\diam}{\operatorname*{diam}}
\newcommand{\kernel}{\operatorname*{kernel}}
\newcommand{\support}{\operatorname*{supp}}

\newcommand{\N}{\mathcal{N}}

\newcommand{\R}{\mathbb{R}}
\newcommand{\IH}{\mathcal{I}_H}

\newcommand{\tnorm}[1]{|||#1|||}
\newcommand{\norm}[1]{\|#1\|}

\newcommand{\Cint}{C_{\IH}}

\newcommand{\Col}{C_{\operatorname*{ol}}}

\newcommand{\Vf}{V_{\operatorname*{f}}}
\newcommand{\vf}{v_{\operatorname*{f}}}
\newcommand{\uf}{u_{\operatorname*{f}}}
\newcommand{\Vc}{V_{\operatorname*{c}}}
\newcommand{\vc}{v_{\operatorname*{c}}}
\newcommand{\uc}{u_{\operatorname*{c}}}
\newcommand{\tuc}{\tilde{u}_{\operatorname*{c}}}
\newcommand{\Pf}{\mathcal{P}_{\operatorname*{f}}}
\renewcommand{\H}{\text{H}}

\title[Computation of eigenvalues by numerical upscaling]{Computation of eigenvalues by numerical upscaling}

\author{Axel M\r{a}lqvist and Daniel Peterseim}
\thanks{A. M\r{a}lqvist is supported by The G\"{o}ran Gustafsson Foundation and The Swedish Research Council.}
\address[D. Peterseim]{Rheinische Friedrich-Wilhelms-Universit\"at Bonn, Institute for Numerical Simulation, Wegelerstr. 6, 53115 Bonn, Germany}
\thanks{D. Peterseim was partially supported by the DFG Research Center Matheon Berlin through project C33.}
\email{peterseim@ins.uni-bonn.de}
\address[A. M\r{a}lqvist]{Department of Mathematical Sciences, Chalmers University of Technology and University of Gothenburg, SE-412 96 G\"{o}teborg, Sweden}

\date{\today}



\begin{document}

\begin{abstract}
We present numerical upscaling techniques for a class of linear second-order self-adjoint elliptic partial differential operators (or their high-resolution finite element discretization). As prototypes for the application of our theory we consider benchmark multi-scale eigenvalue problems in reservoir modeling and material science. We compute a low-dimensional generalized (possibly mesh free) finite element space that preserves the lowermost eigenvalues in a superconvergent way. The approximate eigenpairs are then obtained by solving the corresponding low-dimensional algebraic eigenvalue problem. The rigorous error bounds are based on two-scale decompositions of $\H^1_0(\Omega)$ by means of a certain Cl\'ement-type quasi-interpolation operator.
\end{abstract}

\maketitle

\section{Introduction}
This paper presents and analyzes a novel numerical upscaling technique for computing eigenpairs of self-adjoint linear elliptic second order differential operators with arbitrary positive bounded coefficients. The precise setting of the paper is as follows. Let $\Omega\subset\mathbb{R}^{d}$ be a bounded polyhedral Lipschitz domain and let $A\in L^\infty(\Omega,\R^{d\times d}_{\operatorname*{sym}})$ be a matrix-valued coefficient with uniform spectral bounds $0<\gmin\leq \gmax<\infty$,
\begin{equation}\label{e:spectralbound}
\sigma(A(x))\subset [\gmin,\gmax]
\end{equation}
for almost all $x\in \Omega$.
We want to approximate the eigenvalues of the prototypical operator $-\ddiv(A\grad \bullet)$. The corresponding eigenproblem in variational formulation reads: find pairs consisting of an eigenvalue $\lambda\in \R$ and associated non-trivial eigenfunction $u\in V:=H_{0}^{1}( \Omega)$ such that
\begin{subequations}\label{e:model}
\end{subequations}
\begin{equation}\label{e:modelstrong}
a(u,v):=\int_\Omega (A\nabla u)\cdot \nabla v \dx= \lambda\int_\Omega u v \dx=:\lambda(u,v)_{L^2(\Omega)}\tag{\ref{e:model}}
\end{equation}
for all $v\in V$.
We are mainly interested in the lowermost eigenvalues of \eqref{e:model} or, more precisely, in the lowermost eigenvalues of the discretized problem: find $\lambda_h\in \R$ and associated non-trivial eigenfunctions $u_h\in V_h\subset V$ such that
\begin{equation}\label{e:modelstrongh}
a(u_h,v)=\lambda_h( u_h,v)_{L^2(\Omega)}\quad\text{for all }v\in V_h.  \tag{\ref{e:model}.h}
\end{equation}
Here and throughout the paper, the discrete space $V_h\subset V$ shall be a conforming finite element space of dimension $N_h$ based on some regular finite element mesh $\tri_h$ of width $h$.

Popular approaches for the computation of these eigenvalues include Lanczos/Arnoldi-type iterations (as implemented, e.g., in \cite{MR1621681}) or the QR-algorithm applied directly to the $N_h$-dimensional finite element matrices. If a certain structure of the discretization can be exploited (e.g., a hierarchy of finite element meshes and/or spaces) some preconditioned outer iteration for the eigenvalue approximation may be performed and linear problems are solved (approximately) in every iteration step \cite{MR526484}, \cite{MR1942725}, \cite{MR1991264}; see also \cite{MR2427903} and \cite{MR2425502} and references therein.

Our aim is to avoid the application of any eigenvalue solver to the fine scale discretization \eqref{e:modelstrongh} directly.
We introduce a second, coarser discretization scale $H>h$ instead. On the corresponding coarse mesh $\tri_H$, we compute a generalized finite element space $\Vc$ of dimension $N_H\ll N_h$. The solutions $(\lambda_H,\uc)\in\R\times\Vc$ of
\begin{equation}\label{e:modelstrongH}
a(\uc,v)=\lambda_H (\uc,v)_{L^2(\Omega)}\quad\text{for all }v\in \Vc,  \tag{\ref{e:model}.H}
\end{equation}
then yield accurate approximations of the first $N_H$ eigenpairs of \eqref{e:modelstrongh} and, hence, of the first $N_H$ eigenpairs of \eqref{e:modelstrong} (provided that $V_h$ is properly chosen).

The computation of the coarse space $\Vc$ involves the (approximate) solution of $N_H$ linear equations on the fine scale (one per coarse node). We emphasize that these linear problems are completely independent of each other. They can be computed in  parallel without any communication.

The error $\lambda_H-\lambda_h$ between corresponding eigenvalues of \eqref{e:modelstrongH} and \eqref{e:modelstrongh}, i.e., the error committed by the upscaling from the fine discretization scale $h$ to the coarse discretization scale $H$, is expressed in terms of $H$. Without any assumptions on the smoothness of the eigenfunctions of \eqref{e:modelstrong} or \eqref{e:modelstrongh}, we prove that these errors are at least of order $H^4$. Note that a standard first-order conforming finite element computation on the coarse scale yields accuracy $H^2$ under full $\H^2(\Omega)$ regularity, see e.g.~\cite{MR1770064}.
Since our estimates are both, of high order (at least $H^4$) and independent
of the underlying regularity, the accuracy of our approximation may actually suffice to fall below the error $\lambda_h-\lambda$ of the fine scale discretization which is of order $C h^{2s}$ where both the constant $C$ and the exponent $s\in[0,1]$ depend on the regularity of the data (convexity of $\Omega$, differentiability and variability of $A$) in a crucial way.

The idea of employing a two-level techniques for the acceleration of eigensolvers is not new. The two-grid method of \cite{MR1677419} allows certain post-processing (solution of linear problems on the fine scale). For standard first-order conforming finite element coarse spaces, this technique decreases the eigenvalue error from $H^2$ to $H^4$ (up to fine scale errors as above) if the corresponding eigenfunctions are $\text{H}^2$-regular. The regularity
assumption is essential and not justified on non-convex domains or for heterogeneous and highly variable coefficients. However, the post-processing technique applies as well to the generalized finite element coarse space $\Vc$ and yields eigenvalue errors of order $H^6$ without any regularity assumptions.
 
In cases with singular eigenfunctions (due to re-entrant corners in the domain or isolated jumps of the coefficient), one might as well use modern mesh-adaptive algorithms driven by some a posteriori error estimator as proposed and analyzed, e.g., in \cite{MR1770064}, \cite{MR1909253}, \cite{MR1998821}, \cite{MR2810801}, \cite{MR2531037}, \cite{MR2485445}, \cite{MR2760060}, \cite{doi:10.1137/090769430}, \cite{MR3018645}. We are not competing with these efficient algorithms. However, adaptive mesh refinement has its limitations. For instance, if the diffusion coefficient $A$ is highly variable on microscopic scales, the mesh width has to be sufficiently small to resolve these variations \cite{PS12}. For problems in geophysics or material sciences with characteristic geometric features on microscopic length scales, this so-called resolution condition is often so restrictive that the initial mesh must be chosen very fine and further refinement exceeds computer capacity. Our method is especially designed for such situations which require coarsening rather than refinement.

A particular application of our methodology is the computation ground states of Bose-
Einstein condensates as solutions of the Gross-Pitaevskii equation. Here, certain resolution (small $h$) is required in order to ensure unique solvability of the discrete non-linear eigenvalue problem. It is already exposed in \cite{HMP13} that our upscaling approach leads to a significant speed-up in computational time because the expensive iterative solver for the non-linear eigenproblem needs to be applied solely on a space of very low dimension.

The main tools in this paper are localizable orthogonal decompositions of $\H^1_0(\Omega)$ (or its subspace $V_h$) into coarse and fine parts. These decompositions   are presented in Section~\ref{s:omd}. The two-level method for the approximation of eigenvalues is presented in Section~\ref{s:approx}. Section~\ref{s:error} contains its error analysis. The efficient local approximation of the coarse space, the generalization to non-nested grids, a post-processing technique, and further complexity issues are discussed in Section~\ref{s:practical}. Finally, Section~\ref{s:numexp} demonstrates the performance of the method in numerical experiments.

In the remaining part of this paper, \label{p:lesssim} we will frequently make use of the notation $b_1\lesssim b_2$ which abbreviates $b_1\leq Cb_2$,  with some multiplicative constant $C>0$ which only depends on the domain $\Omega$ and the parameter $\gamma$ (cf. \eqref{e:shapereg} below) that measures the quality of some underlying finite element mesh. We emphasize that the $C$ does \emph{not} depend on the mesh sizes $H$, $h$, the eigenvalues, or the coefficient $A$. Furthermore, $b_1\approx b_2$ abbreviates $b_1\lesssim b_2\lesssim b_1$.

\section{Finite Element Spaces and Quasi-Interpolation}\label{s:pre}
This section presents some preliminaries on finite element meshes, spaces, and interpolation.
\subsection{Finite element mesh}\label{ss:mesh}
We consider two discretization scales $H> h>0$. Let $\tri_H$ (resp. $\tri_h$) denote corresponding regular (in the sense of \cite{CiarletPb}) finite element meshes of $\Omega$ into closed simplices with mesh-size functions $0<H\in L^\infty(\Omega)$ defined by $H\vert_T=\diam T=:H_T$ for all $T\in\tri_H$ (resp. $0<h\in L^\infty(\Omega)$ defined by $h\vert_t=\diam t=:h_t$ for all $t\in\tri_h$). The mesh sizes may vary in space but we will not exploit the possible mesh adaptivity in this paper.

The error bounds, typically, depend on the maximal mesh sizes $\norm{H}_{L^\infty(\Omega)}$. If no confusion seems likely, we will use $H$ also to denote the maximal mesh size instead of writing $\norm{H}_{L^\infty(\Omega)}$.
For the sake of simplicity we assume that $\tri_h$ is derived from $\tri_H$ by some regular, possibly non-uniform, mesh refinement. However, this condition is not essential and Section~\ref{ss:coarsening} will discuss possible generalizations.

As usual, the error analysis depends on the constant $\gamma>0$ which represents the shape regularity of the finite element mesh $\triH$;
\begin{equation}
\label{e:shapereg}
\gamma:=\max_{T\in\triH}\gamma_T\quad\text{with}\quad \gamma_T:=\frac{\diam{T}}{\diam{B_T}}\; \text{ for } T\in\triH,
\end{equation}
where $B_T$ denotes the largest ball contained in $T$.

\subsection{Finite element spaces}\label{ss:fem}
The first-order conforming finite element space corresponding to $\tri_H$ is given by
\begin{equation}\label{e:couranta}
V_H:=\{v\in V\;\vert\;\forall T\in\tri_H,v\vert_T \text{ is a polynomial of total degree}\leq 1\}.
\end{equation}
Let $\N_H$ denote the set of interior vertices of $\tri_H$. For every vertex $z\in\N_H$, let $\phi_z\in V_H$ denote the corresponding nodal basis function (tent/hat function) determined by nodal values
\begin{equation*}
 \phi_z(z)=1\;\text{ and }\;\phi_z(y)=0\quad\text{for all }y\neq z\in\N_H.
\end{equation*}
These nodal basis functions form a basis of $V_H$. The dimension of $V_H$ equals the number of interior vertices,
\begin{equation*}
N_H:=\dim V_H=|\N_H|.                                                                                                                                                                                                                                     \end{equation*}

Let $V_h\supset V_H$ denote some conforming finite element space corresponding to the fine mesh $\tri_h$. It can be the space of continuous piecewise affine functions on the fine mesh or any other (generalized) finite element space that contains $V_H$, e.g., the space of continuous $p$-th order piecewise polynomials as in  \cite{MR2608360}. By $N_h:=\dim V_h$ we denote the dimension of $V_h$. For standard choices of $V_h$, this dimension is proportional to the number of interior vertices in the fine mesh $\tri_h$.

\subsection{Quasi-interpolation}\label{ss:quasi}
The key tool in our construction will be the bounded linear surjective Cl\'ement-type (quasi-)interpolation operator $\IH: \H^1_0(\Omega)\rightarrow V_H$ presented and analyzed in \cite{MR1706735}. Given $v\in \H^1_0(\Omega)$, $\IH v := \sum_{z\in\N_H}(\IH v)(z)\phi_z$ defines a (weighted) Cl\'ement interpolant with nodal values
\begin{equation}\label{e:clement}
 (\IH v)(z):=\frac{(v, \phi_z)_{L^2(\Omega)}}{(1,\phi_z)_{L^2(\Omega)}}
\end{equation}
for $z\in\N_H$. The nodal values are weighted averages of the function over nodal patches $\omega_z:=\support\phi_z$.
Recall the (local) approximation and stability properties of the interpolation operator $\IH$ \cite{MR1706735}: There exists a generic  constant $\Cint$ such that for all $v\in \H^1_0(\Omega)$ and for all $T\in\triH$ it holds
\begin{equation}\label{e:interr}
 H_T^{-1}\|v-\IH v\|_{L^{2}(T)}+\|\nabla(v-\IH v)\|_{L^{2}(T)}\leq \Cint \| \nabla v\|_{L^2(\omega_T)},
\end{equation}
where $\omega_T:=\cup\{K\in\triH\;\vert\;T\cap K\neq\emptyset\}$. The constant $\Cint$ depends on the shape regularity parameter $\gamma$ of the finite element mesh $\triH$ (see \eqref{e:shapereg} above) but not on $H_T$.

Note that there exists a constant $\Col>0$ that only depends on $\gamma$ such that the number of elements covered by $\omega_T$ is uniformly bounded (w.r.t. $T$)
by $\Col$,
\begin{equation}\label{e:Col}
 \max_{T\in\tri_H}\left|\{K\in\triH\;\vert\;K\subset \omega_T\}\right|\leq\Col.
\end{equation}
Both constant, $\Cint$ and $\Col$, may be hidden in the notation ``$\lesssim$'' introduced at the end of the Introduction on page \pageref{p:lesssim}.

\section{Two-scale Decompositions}\label{s:omd}
Two-scale decompositions of functions $u\in V_h$ into some macroscopic/coarse part $\uc$ plus some microscopic/fine part $\uf$ with a certain orthogonality relation are at the very heart of this paper. The macroscopic or coarse part will be an element of a low-dimensional (classical or generalized) finite element space based on some coarse finite element mesh. The microscopic or fine part may oscillate on fine scales that cannot be represented on the coarse mesh.

We stress that all subsequent results are valid even if $h=0$, i.e., if $V_h$ is replaced with $V=H^1_0(\Omega)$. Actually, the structure of $V_h$ being the space of continuous piecewise polynomials is never exploited. As far as the theory is concerned, $V_h$ could be any space (finite or infinite dimensional) that satisfies $V_H\subsetneq V_h\subseteq H^1_0(\Omega)$. 

The initial coarse space $V_H$ may as well be generalized. This will be discussed in Section~\ref{ss:coarsening}.

\subsection{$L^2$-orthogonal two-scale decomposition}\label{ss:decomp}
We define the fine scale space
$$\Vf:=\kernel \bigl(\IH\vert_{V_h}\bigr)\subset V_h,$$
which will take over the role of the microscopic/fine part in all subsequent decompositions.

Our particular choice of a quasi-interpolation operator gives rise to the following orthogonal decomposition. 
Remember that $(\bullet,\bullet)_{L^2(\Omega)}:=\int_\Omega \bullet\bullet\dx$ abbreviates the canonical scalar product in $L^2(\Omega)$ and let  $\norm{\bullet}:=\sqrt{(\bullet,\bullet)_{L^2(\Omega)}}$ abbreviate the corresponding norm of $L^2(\Omega)$.
\begin{lemma}[$L^2$-orthogonal two-scale decomposition]\label{l:L2od}
Any function $u\in V_h$ can be decomposed uniquely into the sum of $u_H:=\IH\vert_{V_H}^{-1}(\IH u)\in V_H$ and $\uf:=u-u_H\in\Vf$ with
\begin{equation}\label{e:ortho}
 (u_H,\uf)_{L^2(\Omega)}=0.
\end{equation}
The orthogonality implies stability in the sense of
\begin{equation*}
 \norm{u_H}^2+\norm{\uf}^2=\norm{u}^2.
\end{equation*}
\end{lemma}
\begin{proof}[Proof of Lemma~\ref{l:L2od}]
It is easily verified that the restriction of $\IH$ on the finite element space $V_H$ is invertible. This yields the decomposition.

For the proof of orthogonality, let $v_H=\sum_{z\in\N_H}v_H(z)\phi_z\in V_H$ and $\vf\in\Vf$ be arbitrary. Since $\IH\vf=0$, we have that $(\phi_z, \vf)_{L^2(\Omega)}=(\IH\vf)(z)\int_\Omega\phi_z\dx=0$ for all $z\in\N_H$. This yields
$$(v_H, \vf)_{L^2(\Omega)} =\sum_{z\in\N_H}v_H(z)(\phi_z, \vf)_{L^2(\Omega)}=0$$
and shows that $V_H$ and $\Vf$ are orthogonal subspaces of $V_h$.
\end{proof}
We may rewrite Lemma~\ref{l:L2od} as
\begin{equation}\label{e:l2od}
 V_h=V_H\oplus \Vf\;\text{ and }\; (V_H,\Vf)_{L^2(\Omega)}=0.
\end{equation}
\begin{remark}[$L^2$-projection onto the finite element space]
Note that the operator $\IH$ is well-defined as a mapping from $L^2(\Omega)$ onto $V_H$. In particular, it is stable in the sense that for any $v\in L^2(\Omega)$, it holds that $\|\IH v\|\lesssim \|v\|$. From the arguments of Lemma~\ref{l:L2od} one easily verifies that the $L^2$-orthogonal projection $\Pi^{L^2}_{V_H}:L^2(\Omega)\rightarrow V_H$ onto the finite element space $V_H$ may be characterized via the modified Cl\'ement interpolation \eqref{e:clement},
\begin{equation*}
 \Pi^{L^2}_{V_H}= \IH\vert_{V_H}^{-1} \IH.
\end{equation*}
Furthermore, it holds $\Vf=\kernel\bigl(\Pi^{L^2}_{V_H}\vert_{V_h}\bigr)$, i.e., $\Vf$ might as well be characterized via $\Pi^{L^2}_{V_H}$. This does not change the method. For theoretical purposes, we prefer to work with $\IH$ because it is a local operator.
\end{remark}

\subsection{$a$-orthogonal two-scale decomposition}
The orthogonalization of the decomposition \eqref{e:l2od} with respect to the scalar product $a(\bullet,\bullet):=\int_\Omega (A\nabla\bullet)\cdot\nabla\bullet\dx$ yields the definition of a generalized finite element space $\Vc$, that is the $a$-orthogonal complement of $\Vf$ in $V_h$. Given $v\in V_h$, define the $a$-orthogonal fine scale projection operator $\Pf v\in\Vf$ by
\begin{equation*}\label{e:finescaleproj}
 a(\Pf v,w)=a(v,w)\quad\text{for all }w\in \Vf.
\end{equation*}
We define the energy norm $\tnorm{\bullet}:=\sqrt{a(\bullet,\bullet)}$ (the norm induced by the scalar product $a$).
\begin{lemma}[$a$-orthogonal two-scale decomposition]\label{l:aod}
Any function $u\in V_h$ can be decomposed uniquely into $u=\uc+\uf$, where
$$\uc:=(1-\Pf)u\in(1-\Pf)V_H=:\Vc$$
and $$\uf:=\Pf u\in\Vf=\kernel (\IH\vert_{V_h}).$$
The decomposition is orthogonal
\begin{equation}\label{e:orthoa}
 a(\uc,\uf)=0,
\end{equation}
and, hence, stable in the sense of
\begin{equation}\label{e:stab}
 \tnorm{\uc}^2+\tnorm{\uf}^2=\tnorm{u}^2.
\end{equation}
\end{lemma}

In other words,
\begin{equation}\label{e:aod}
 V_h=\Vc\oplus \Vf\;\text{and }\; a(\Vc,\Vf)=0.
\end{equation}

We shall emphasize at this point that the decompositions in Lemma~\ref{l:L2od} and Lemma~\ref{l:aod} are different in general. In particular, the fine scale part $\vf$ may not be the same.

The orthogonalization procedure (with respect to $a(\bullet,\bullet)$) does not preserve the $L^2$-orthogonality. However, the key observation of this section is that the resulting decomposition \eqref{e:aod} is almost orthogonal in $L^2(\Omega)$.
\begin{theorem}[$L^2$-quasi-orthogonality of the $a$-orthogonal decomposition]\label{t:quasiL2}
The decomposition $V_h=\Vc\oplus \Vf$ from Lemma~\ref{l:aod} is $L^2$-quasi-orthogonal in the sense that for all $\vc\in\Vc$  and all $\vf\in\Vf$, it holds
\begin{equation}\label{e:orthoquasi}
 (\vc,\vf)_{L^2(\Omega)}\lesssim H^2\norm{\nabla\vc}\norm{\nabla\vf}\leq \alpha^{-1}H^2\tnorm{\vc}\tnorm{\vf}.
\end{equation}
The decomposition is stable in the sense that
\begin{equation}\label{e:stabquasi}
 \norm{\vc}^2+\norm{H^{-1}\vf}^2\lesssim \alpha^{-1}\tnorm{\vc+\vf}^2.
\end{equation}
\end{theorem}
\begin{proof}
Given any $\vc\in\Vc$ and $\vf\in\Vf$,
Lemma~\ref{l:L2od} implies that $$(\IH\vc,\vf)_{L^2(\Omega)}=0.$$ Since $\IH\vf =0$, the Cauchy-Schwarz inequality, \eqref{e:interr}, and \eqref{e:Col} yield
\begin{align}\label{e:stabquasi1}
 (\vc,\vf)_{L^2(\Omega)}&=(\vc-\IH\vc,\vf-\IH\vf)_{L^2(\Omega)}\lesssim H^2\norm{\nabla\vc}\norm{\nabla\vf}.
\end{align}
This is the quasi-orthogonality. The same arguments show that
\begin{align*}
 (H^{-1}\vf,H^{-1}\vf)_{L^2(\Omega)}&=\left(H^{-1}(\vf-\IH\vf),H^{-1}(\vf-\IH\vf)\right)_{L^2(\Omega)}\\
&\lesssim\sum_{T\in\tri_H} \norm{\nabla\vf}_{L^2(\omega_T)}^2\\
&\lesssim \alpha^{-1}\tnorm{\vf}^2.
\end{align*}
This, Friedrichs' inequality
\begin{align*}
 \norm{\vc}&\leq \pi^{-1}\diam\Omega \norm{\nabla\vc},
\end{align*}
and \eqref{e:stab} readily prove the stability estimate.
\end{proof}

\section{Upscaled Approximation of Eigenvalues and Eigenfunctions}\label{s:approx}
This section presents a new scheme for the approximation of eigenvalues and eigenfunctions of \eqref{e:modelstrongh} or \eqref{e:modelstrong}. Section~\ref{ss:varfine} recalls the variational formulation and some characteristic properties of the problem. The new upscaled approximation is then introduced in Section~\ref{ss:varcoarse}.
\subsection{Variational formulation and fine scale discretization}\label{ss:varfine}
For problem \eqref{e:modelstrong}, there exists a countable number of eigenvalues $\lambda^{(\ell)}$ ($\ell\in\mathbb{N}$) and corresponding eigenfunctions $u^{(\ell)}\in V$. Recall their characterization as solutions of the variational problem
\begin{equation}\label{e:modelvar}
a(u^{(\ell)},v) = \lambda^{(\ell)}(u^{(\ell)},v)_{L^2(\Omega)}\quad\text{for all }v\in V.
\end{equation}
Since $A$ is symmetric, all eigenvalues are real and positive. They can be sorted ascending
$$0<\lambda^{(1)}\leq\lambda^{(2)}\leq\lambda^{(3)}\leq\ldots\;.$$
Depending on the actual domain $\Omega$ and the coefficient $A$, there may be multiple eigenvalues. A multiple eigenvalue is repeated several times according to its multiplicity in the enumeration above.
Let $u^{(\ell)}$ ($\ell\in\mathbb{N}$) be normalized to one in $L^2(\Omega)$, i.e.,  $\|u^{(\ell)}\| = 1$. It is well known that the eigenfunctions enjoy (or, in the case of multiple eigenvalues, may be chosen such that they fulfill) the orthogonality contraints
\begin{equation}\label{e:orthogonalitiesdisc}
 a(u^{(\ell)},u^{(m)}) = (u^{(\ell)},u^{(m)})_{L^2(\Omega)} = 0\quad\text{ if } \ell \neq m.
\end{equation}

The Rayleigh-Ritz discretization of \eqref{e:modelvar} with respect to the fine scale finite element space $V_h$ reads: find $\lambda_h^{(\ell)}\in\mathbb{R}$ and non-trivial $u_h^{(\ell)}\in V_h$ such that
\begin{equation}\label{e:modeldisch}
a(u_h^{(\ell)},v)=\lambda_h^{(\ell)}(u_h^{(\ell)},v)_{L^2(\Omega)}\quad\text{for all }v\in V_h.
\end{equation}
Since $V_h$ is a finite-dimensional subspace of $V$, we can order the discrete eigenvalues similar as the original ones
$$0<\lambda_h^{(1)}\leq\lambda_h^{(2)}\leq\lambda_h^{(3)}\leq\cdots\leq\lambda_h^{(N_h)}.$$
Again, multiple eigenvalues are repeated according to their multiplicity.
Let $u_h^{(\ell)}$ ($\ell=1,2,\ldots,N_h$) be normalized to one in $L^2(\Omega)$, i.e.,  $\|u_h^{(\ell)}\| = 1$. The discrete eigenfunctions satisfy (or, in the case of multiple eigenvalues, can be chosen such that they satisfy) the orthogonality contraints
\begin{equation}\label{e:orthogonalitiesh}
 a(u_h^{(\ell)},u_h^{(m)}) = (u_h^{(\ell)}, u_h^{(m)} )_{L^2(\Omega)} = 0\quad\text{ if } \ell \neq m.
\end{equation}
We do not intend to solve the fine scale eigenproblem \eqref{e:modeldisch}. We aim to approximate its eigenpairs $(\lambda_h^{(\ell)},u_h^{(\ell)})$ with the help of the coarse space $\Vc$ defined in Lemma~\ref{l:aod}.

\subsection{Coarse scale discretization}\label{ss:varcoarse}
Recall the definition of the coarse space
\begin{equation*}
 \Vc:=(1-\Pf)V_H
\end{equation*}
from Lemma~\ref{l:aod}. This means that $\Vc$ is the image of $V_H$ under the projection operator $1-\Pf$, where $\Pf$ is the $a$-orthogonal projection onto the space
\begin{equation*}
 \Vf:=\{v\in V_h\,\vert\,\IH v = 0\}.
\end{equation*}
Since the intersection of $V_H$ and $\Vf$ is the trivial subspace (cf. Lemma~\ref{l:L2od}), it holds
\begin{equation*}
\dim\Vc=\dim V_H=N_H.
\end{equation*}
Moreover, the images of the nodal basis functions $\phi_z$ ($z\in\N_H$) under $(1-\Pf)$ yield a basis of $\Vc$,
\begin{equation}\label{e:basisVc}
 \Vc = \operatorname{span}\{(1-\Pf)\phi_z\;\vert\;z\in\N_H\}.
\end{equation}

In order to actually compute those basis functions, we need to approximate $N_H$ solutions $\psi_z=\Pf\phi_z\in\Vf$ of
\begin{equation}\label{e:corrector}
 a(\psi_z,v)=a(\phi_z,v)\quad\text{for all }v\in\Vf.
\end{equation}
These problems are linear. The only difference to a standard Poisson problem is that there are some linear constraints hidden in the space $\Vf$, that is, the quasi-interpolation of trial and test functions vanishes. In practice, these constraints are realized using Lagrange multipliers.

The linear problems \eqref{e:corrector} may be solved in parallel. Moreover, Section~\ref{ss:compression} below will show that these linear problems may be restricted to local subdomains of diameter $\approx|\log(H)| H$ centered around the coarse vertex $z$, so that the complexity of solving all corrector problems exceeds to the cost of solving one linear Poisson problem on the fine mesh only by a factor that depends algebraically on $|\log(H)|$.

The Rayleigh-Ritz discretization of \eqref{e:modeldisch} (and \eqref{e:modelvar}) with respect to the generalized finite element space $\Vc$ reads: find $\lambda_H^{(\ell)}\in\mathbb{R}$ and non-trivial $\uc^{(\ell)}\in\Vc$ such that
\begin{equation}\label{e:modeldisc}
a(\uc^{(\ell)},v)=\lambda_H^{(\ell)}(\uc^{(\ell)},v)_{L^2(\Omega)}\quad\text{for all }v\in \Vc.
\end{equation}
The assembly of the corresponding finite element stiffness and mass matrices requires only the evaluation of the corrector functions $\psi_z=\Pf\phi_z\in\Vf$ computed previously. In genereal, these matrices are not sparse. However, either the dimension of the coarse problem $N_H\ll N_h$ is so small that the lack of sparsity is not an issue or the matrices may be approximated by sparse matrices with negligible loss of accuracy (see Section~\ref{ss:compression} below).

The discrete eigenvalues are ordered (multiple eigenvalues are repeated according to their multiplicity)
$$0<\lambda_H^{(1)}\leq\lambda_H^{(2)}\leq\lambda_H^{(3)}\leq\cdots\leq\lambda_H^{(N_H)}.$$
Let also $\uc^{(\ell)}$ ($\ell=1,2,\ldots,N_H$) be normalized to one in $L^2(\Omega)$, i.e.,  $(\uc^{(\ell)},\uc^{(\ell)})_{L^2(\Omega)} = 1$. The discrete eigenfunctions satisfy (or, in the case of multiple eigenvalues, can be chosen such that they satisfy) the orthogonality contraints
\begin{equation}\label{e:orthogonalities}
 a(\uc^{(\ell)},\uc^{(m)}) = (\uc^{(\ell)}, \uc^{(m)} )_{L^2(\Omega)} = 0\quad\text{ if } \ell \neq m.
\end{equation}

\section{Error analysis}\label{s:error}
In the subsequent paragraphs we will present error bounds for the approximate eigenvalues and eigenfunctions based on the variational techniques from \cite{MR0443377} (which are based on \cite{MR0203926} on their part); see also \cite{MR2652780}.

\subsection{Two-scale decomposition revisited}
The eigenfunctions allow a different (with respect to Section~\ref{s:omd}) characterization of a macroscopic function, that is, any function spanned by eigenfunctions related to the $\ell$ lowermost eigenvalues.
Define
\begin{equation}\label{e:El}
 E_{\ell}:=\operatorname{span}\{u_h^{(1)},\ldots,u_h^{(\ell)}\}.
\end{equation}
We will have a closer look at the quasi-orthogonality result of Lemma~\ref{l:aod} given some macroscopic function $u\in E_{\ell}$.
\begin{lemma}[$L^2$-quasi-orthogonality of the $a$-orthogonal decomposition of macroscopic functions]\label{c:L2od}
Let $\ell\in\mathbb{N}$ and let $u=\uc+\uf\in E_{\ell}$ with $\norm{u}=1$, where $\uc\in\Vc$ (resp. $\uf\in\Vf$) denotes the coarse scale part (resp. fine scale part) of $u$ according to the $a$-orthogonal decomposition in  Lemma~\ref{l:aod}.
Then it holds
\begin{eqnarray}
\tnorm{\uc}&\leq& \sqrt{\lambda_h^{(\ell)}},\label{e:quasi1}\\
\tnorm{\uf}&\lesssim& \ell\;\frac{\bigl(\lambda_h^{(\ell)}\bigr)^{3/2}}{\alpha}\;H^2,\;\text{ and}\label{e:quasi2}\\
 |(\uc,\uf)_{L^2(\Omega)}|&\lesssim&\ell\left(\frac{\lambda_h^{(\ell)}}{\alpha}\right)^2H^4.\label{e:quasi3}
\end{eqnarray}
\end{lemma}
\begin{proof}
Let $\delta_j\leq 1$, $j=1,2,\ldots,\ell$, be the coefficients in the representation of $u$ by eigenfunctions, that is, $u=\sum_{j=1}^\ell \delta_j u_h^{(j)}$. Then \eqref{e:quasi1} follows from the fact that $(1-\Pf)$ is a projection and the obvious bound $\tnorm{u}^2\leq \lambda_h^{(\ell)}$.

For the proof of \eqref{e:quasi2}, we employ some algebraic manipulations and equation \eqref{e:modeldisch},
\begin{equation}\label{e:est1}
 \tnorm{\uf}^2=a(u,\uf)=\sum_{j=1}^\ell \delta_j a(u_h^{(j)},\uf)=\sum_{j=1}^\ell \delta_j \lambda_h^{(j)} (u_h^{(j)},\uf)_{L^2(\Omega)}.
\end{equation}
Lemma~\ref{l:L2od}, the Cauchy-Schwarz inequality, \eqref{e:interr}, and \eqref{e:Col} yield
\begin{equation}\label{e:est2}
 (u_h^{(j)},\uf)_{L^2(\Omega)}\lesssim \alpha^{-1}H^2\tnorm{u_h^{(j)}}\tnorm{\uf}
\end{equation}
(cf. \eqref{e:stabquasi1}). The combination of \eqref{e:est1}-\eqref{e:est2}, $\tnorm{u_h^{(j)}}^2=\lambda_h^{(j)}\leq\lambda_h^{(\ell)}$ and
 $\delta_j\leq 1$ yields the upper bound of $\tnorm{\uf}$.

The inequality \eqref{e:quasi3} follows readily from Theorem~\ref{t:quasiL2} and the bounds \eqref{e:quasi1}--\eqref{e:quasi2}.
\end{proof}
\begin{remark}[Improved $L^2$-quasi-orthogonality under regularity]\label{r:smooth} Consider the full space space $V_h=V$. Then, in certain cases, e.g., if $\Omega$ is convex and the coefficient $A$ is constant, we have that any macroscopic function $u\in E_{\ell}$ is in $\H^2(\Omega)$ and $\norm{\nabla^2 u}\lesssim\lambda^{(\ell)}/\alpha\norm{u}$. Such an instance of regularity gives rise to an additional power of $H\lambda^{(\ell)}/\alpha$ in the estimates \eqref{e:quasi2} and \eqref{e:quasi3} in Lemma~\ref{c:L2od}.
This is due to the approximation property
\begin{equation}\label{e:interr2}
 \|v-\IH v\|\lesssim H^2\| v\|_{\text{H}^2(\Omega)}
\end{equation}
for $v\in V\cap\H^2(\Omega)$, and the possible modification
\begin{equation*}
 (u^{(j)},\uf)_{L^2(\Omega)}=(u^{(j)}-\IH u^{(j)},\uf-\IH\uf)_{L^2(\Omega)}\lesssim \frac{H^3\lambda^{(j)}}{\alpha^{2}}\tnorm{\uf}
\end{equation*}
of \eqref{e:est2}.
\end{remark}

\subsection{Estimates for approximate eigenvalues}
We first introduce the Rayleigh quotient, which is defined for non-trivial $v\in V_h$ by
\begin{equation*}
 R(v):=\frac{a(v,v)}{(v,v)}.
\end{equation*}
Recall that the $\ell$th eigenvalue of \eqref{e:modeldisch} is characterized via the minmax-principle (which goes back to Poincar\'e \cite{MR1505534})
\begin{equation}\label{e:minmax}
 \lambda_h^{(\ell)}=\min_{S\in\mathcal{S}_\ell(V_h)}\max_{v\in S\setminus\{0\}} R(v),
\end{equation}
where $\mathcal{S}_\ell(V)$ denotes the set of $\ell$-dimensional subspaces of $V_h$.
This principle applies equally well to the coarse problem \eqref{e:modeldisc}, i.e.,
\begin{equation}\label{e:minmaxdisc}
 \lambda_H^{(\ell)}=\min_{S\in\mathcal{S}_\ell(\Vc)}\max_{v\in S\setminus\{0\}} R(v)
\end{equation}
characterizes the $\ell$th discrete eigenvalue ($\ell\leq N_H$).
The conformity $\Vc\subset V_h(\subseteq V)$ yields monotonicity
\begin{equation}\label{e:monotonicity}
(\lambda^{(\ell)}\leq\;)\; \lambda_h^{(\ell)}\leq \lambda_H^{(\ell)}\quad\text{for all }\ell=1,2,\ldots,N_H.
\end{equation}
The following theorem gives an estimate in the opposite direction.
\begin{theorem}[Bound for the eigenvalue error]\label{t:lower}
 Let $H$ be sufficiently small so that $H\lesssim \ell^{-1/4}\sqrt{\frac{\alpha}{\lambda^{(\ell)}_h}}$. Then it holds that
\begin{equation}\label{e:eigenvalueest}
 \frac{\lambda_H^{(\ell)}-\lambda_h^{(\ell)}}{\lambda_h^{(\ell)}}\lesssim \ell\left(\frac{\lambda_h^{(\ell)}}{\alpha}\right)^2H^4\quad\text{for all }\ell=1,2,\ldots,N_H.
\end{equation}
\end{theorem}
\begin{proof}
Recall the definition of $E_{\ell}$ in \eqref{e:El} and
define
\begin{equation*}
 \sigma_H^{(\ell)}:=\max_{u\in E_{\ell}:(u,u)_{L^2(\Omega)}=1}|(\uf,\uf)_{L^2(\Omega)}+2(\uc,\uf)_{L^2(\Omega)}|,
\end{equation*}
where $\uc\in\Vc$ (resp. $\uf\in\Vf$) denotes the coarse scale part (resp. fine scale part) of $u\in E_{\ell}$ according to the $a$-orthogonal decomposition in  Lemma~\ref{l:aod}.
The $L^2$-norm of $\uf$ satisfies the estimate
\begin{eqnarray*}
\norm{\uf}^2&=& (u,\uf)_{L^2(\Omega)}-(\uc,\uf)_{L^2(\Omega)}\\&=&(u-\IH u,\uf-\IH \uf)_{L^2(\Omega)}-(\uc,\uf)_{L^2(\Omega)}\\&\lesssim&\ell\left(\frac{\lambda_h^{(\ell)}}{\alpha}\right)^{2}H^4 + |(\uc,\uf)_{L^2(\Omega)}|,
\end{eqnarray*}
which follows from Lemma~\ref{l:L2od}, \eqref{e:interr}, and \eqref{e:Col}. Hence, Lemma~\ref{c:L2od} shows that
\begin{equation*}
 \sigma_H^{(\ell)}\lesssim \ell\left(\frac{\lambda_h^{(\ell)}}{\alpha}\right)^2 H^4.
\end{equation*}
If $H$ is chosen small enough so that $\sigma_H^{(\ell)}\leq \tfrac{1}{2}$ (i.e., $H\lesssim \ell^{-1/4}\sqrt{\frac{\alpha}{\lambda^{(\ell)}}}$), then Lemma~6.1 in \cite{MR0443377} shows that
\begin{equation*}
\lambda_H^{(\ell)}\leq (1-\sigma_H^{(\ell)})^{-1}\lambda_h^{(\ell)}\leq (1+2\sigma_H^{(\ell)})\lambda_h^{(\ell)}.
\end{equation*}
Inserting our estimate for $\sigma_H^{(\ell)}$ readily yields the assertion.
\end{proof}
The triangle inequality allows to control the approximation error with respect to the continuous eigenvalues \eqref{e:modelvar} by
\begin{equation*}
 \lambda_H^{(\ell)}-\lambda^{(\ell)}\lesssim \lambda_h^{(\ell)}-\lambda^{(\ell)} + \ell\;\frac{\bigl(\lambda_h^{(\ell)}\bigr)^3}{\alpha^2}\;H^4.
\end{equation*}
The first term $\lambda_h^{(\ell)}-\lambda^{(\ell)}$ depends on the choice of the space $V_h$ and the regularity of corresponding eigenfunctions in the usual way.

\begin{remark}[Improved eigenvalue error bound for smooth eigenfunction] With regard to Remark~\ref{r:smooth}, the error bound in Theorem~\ref{t:lower} may be improved in the ideal case $V=V_h$ provided that the first $\ell$ eigenfunctions are regular in the sense of $\norm{\nabla^2 u^{(j)}}\lesssim\lambda^{(j)}/\alpha$. The improved bound reads
\begin{equation}\label{e:eigenvalueestsmooth}
 \frac{\lambda_H^{(\ell)}-\lambda^{(\ell)}}{\lambda^{(\ell)}}\lesssim \ell\left(\frac{\lambda^{(\ell)}}{\alpha}\right)^3H^5\quad\text{for all }\ell=1,2,\ldots,N_H.
\end{equation}
This improved bound applies also to the case where $V_h$ is a finite element space if $h$ is sufficiently small.

The improved bound might still be pessimistic in the sense that the error in the $\ell$th eigenvalue/vector depends on the regularity of all previous eigenfunctions. The recent theory \cite{MR2206452} shows that this is not necessarily true.
Moreover, there might be smoothness also in the single summands of the two-scale decomposition which is not exploited.
\end{remark}

\subsection{Estimates for approximate eigenfunctions}
We turn to the error in the approximate eigenfunctions. Again, we follow the receipt provided in \cite{MR0443377}.
\begin{theorem}[Bound for the eigenfunction error]\label{t:eigenvector}
Let $\lambda_h^{(\ell)}$ be an eigenvalue of multiplicity $r$, i.e.,
$\lambda_h^{(\ell)}=\dots =\lambda_h^{(\ell+r-1)}$ with corresponding
eigenspace spanned by the orthonormal basis
$\{u_h^{(\ell+j)}\}_{j=0}^{r-1}$. Let the pairs
$\left(\lambda_H^{(\ell)},\uc^{(\ell)}\right),\ldots,\left(\lambda_H^{(\ell+r-1)},\uc^{(\ell+r-1)}\right)$
be the Rayleigh-Ritz approximations solving equation (\ref{e:modeldisc})
with $\|\uc^{(\ell+j)}\|=1$ for $j=0,1,\ldots,r-1$. If $\ell+r-1\leq N_H$ and if $H\lesssim \ell^{-1/3}(1+\rho)^{-1/3}\sqrt{\alpha/\lambda^{(\ell)}_h}$ is sufficiently small, then 
there exist an orthonormal basis of
$\operatorname{span}(\{u_h^{(\ell+j)}\}_{j=0}^{r-1})$, let us denote the
basis functions $\tilde{u}_h^{(\ell+j)}$, such that for all
$j=0,1,\ldots,r-1$,
\begin{eqnarray}\label{e:eigenvector1}
\qquad\tnorm{\tilde{u}_h^{(\ell+j)}-\uc^{(\ell+j)}}&\lesssim&
\sqrt{\ell}\;\frac{(\lambda_h^{(\ell)})^{3/2}}{\alpha}\;H^2+\ell (1+\rho)\frac{(\lambda_h^{(\ell)})^2}{\alpha^{3/2}}\;H^3,\\
\|\tilde{u}_h^{(\ell+j)}-\uc^{(\ell+j)}\|&\lesssim& \ell(1+\rho) \left(\frac{\lambda_h^{(\ell)}}{\alpha}\right)^{3/2}H^3,\label{e:eigenvector2}
\end{eqnarray}
where $\rho:=\max_{j\not\in \{\ell,\ell+1,\dots,\ell+r-1\}}\frac{\lambda_h^{(\ell)}}{|\lambda_h^{(\ell)}-\lambda_H^{(j)}|}$.
\end{theorem}
\begin{proof}
The analysis presented in \cite[Lemma 6.4 and Theorem 6.2]{MR0443377} shows that, for any $j=0,1,\ldots,r-1$, there is a function $\tuc^{(\ell+j)}\in \operatorname{span}(\{\uc^{(\ell+i)}\}_{i=0}^{r-1})$ such that
\begin{equation*}
\|u_h^{(\ell+j)}-\tuc^{(\ell+j)}\|\leq (1+\rho)\|\Pf u_h^{(\ell+j)}\|.
\end{equation*}
According to the $a$-orthogonal decomposition in Lemma \ref{l:aod}, $\Pf u_h^{(\ell+j)}$ is the fine scale part of $u_h^{(\ell+j)}$. Hence, the interpolation error estimate \eqref{e:interr} and Lemma \ref{c:L2od} yield
\begin{equation*}
\sum_{j=1}^{r-1}\|u_h^{(\ell+j)}-\tuc^{(\ell+j)}\|^2\lesssim (1+\rho)^2\ell^2\left(\frac{\lambda_h^{(\ell)}}{\alpha}\right)^{3}H^6.
\end{equation*}
If the right-hand side is small enough, i.e., if the multiplicative constant hidden in $H\lesssim \ell^{-1/3}(1+\rho)^{-1/3}\sqrt{\alpha/\lambda^{(\ell)}_h}$ is sufficiently small, the linear transformation of the orthonormal basis $\left\{\uc^{(\ell+j)}\right\}_{j=0}^{r-1}$ which defines the set of functions $\left\{\tuc^{(\ell+j)}\right\}_{j=0}^{r-1}$ may be replaced with an orthogonal transformation, without any harm to the estimate. In this regime, the application of the inverse orthogonal transformation to the errors proves the $L^2$ bound \eqref{e:eigenvector2}. 

For the proof of \eqref{e:eigenvector1}, observe that for any $v\in \operatorname{span}(\{u_h^{(\ell+i)}\}_{i=0}^{r-1})$ with $\|v\|=1$ it holds
\begin{align}
\tnorm{v-\uc^{(\ell)}}^2&=\lambda_h^{(\ell)}-2\lambda_h^{(\ell)}(v,\uc^{(\ell)})_{L^2(\Omega)}+\lambda_H^{(\ell)}\nonumber \\
&=\lambda_h^{(\ell)}(2-2(v,\uc^{(\ell)})_{L^2(\Omega)})+\lambda_H^{(\ell)}-\lambda_h^{(\ell)}\nonumber \\
&=\lambda_h^{(\ell)}\|v-\uc^{(\ell)}\|^2+\lambda_H^{(\ell)}-\lambda_h^{(\ell)}. \label{e:enl2}
\end{align}
The assertion then follows by combining equation (\ref{e:enl2}) with $v=\tilde{u}_h^{(\ell+j)}$, \eqref{e:eigenvector2}, and Theorem~\ref{t:lower}.

\end{proof}

\section{Practical Aspects}\label{s:practical}
This section discusses the efficient approximation of the corrector functions $\Pf\phi_z$ from \eqref{e:corrector} by localization, the generalization to non-nested meshes, some post-processing technique, and the overall complexity of our method.

\subsection{Localization of fine scale computations}\label{ss:compression}
The construction of the coarse space $\Vc$ is based on the fine scale equations \eqref{e:corrector} which are formulated on the whole domain $\Omega$. This makes them expensive to compute. However, in \cite{MP11} it was shown that $\Pf\phi_z$ decays exponentially fast outside of the support of the coarse basis function $\phi_z$. We specify this feature as follows. Let $k\in\mathbb{N}$. We define nodal patches $\omega_{z,k}$ of $k$ coarse grid layers centered around the node $z\in\N_H$ by
\begin{equation}\label{e:omega}
 \begin{aligned}
 \omega_{z,1}&:=\support \phi_z=\cup\left\{T\in\triH\;|\;z\in T\right\},\\
 \omega_{z,k}&:=\cup\left\{T\in\triH\;|\;T\cap \omega_{z,k-1}\neq\emptyset\right\} \quad \mbox{for} \enspace k\geq 2.
\end{aligned}
\end{equation}
The result in the decay of $\Pf \phi_z$ in \cite{MP11} can be expressed as follows.
For all vertices $z\in\N_H$ and for all $k\in\mathbb{N}$, it holds
\begin{equation}\label{l:decay}
\norm{A^{1/2}\nabla\Pf\phi_z}_{L^2(\Omega\setminus\omega_{z,k})}\lesssim e^{-(\alpha/\beta)^{1/2} k}\tnorm{\Pf\phi_z}.
\end{equation}
For moderate contrast $\beta/\alpha$, this motivates the truncation of the computations of the basis functions to local patches $\omega_{z,k}$. We approximate $\psi_z=\Pf\phi_z\in\Vf$ from \eqref{e:corrector} with $\psi_{z,k}\in\Vf(\omega_{z,k}):=\{v\in\Vf\;\vert\;v\vert_{\Omega\setminus\omega_{x,k}}=0\}$ such that
\begin{equation}\label{e:correctorlocal}
 a(\psi_{z,k},v)=a(\phi_z,v)\quad\text{for all }v\in\Vf(\omega_{z,k}).
\end{equation}
We emphasize that $$\Vf(\omega_{z,k}) = \{v\in V_h\;\vert\;v\vert_{\Omega\setminus\omega_{x,k}}=0,\;\forall y\in\N_H\cap\omega_{z,k}:(v,\phi_y)_{L^2(\Omega)}=0\},$$ i.e.,
in a practical computation with lagrangian multipliers only one linear constraint per coarse vertex in the patch $\omega_{x,k}$ needs to be considered.

The localized computations yield a modified coarse space $\Vc^k$ with a local basis
\begin{equation}\label{e:basisVck}
 \Vc^k = \operatorname{span}\{\phi_z-\psi_{z,k}\;\vert\;z\in\N_H\}.
\end{equation}
The number of non-zero entries of the corresponding finite element stiffness and mass matrix is proportional to $k^d N_H$ (note that we expect $N_H^2$ non-zero entries without the truncation).
Due to the exponential decay, the very weak condition $k\approx |\log{H}|$ implies that the perturbation of the ideal method due to this truncation is of higher order and the estimates in Theorems~\ref{t:lower} and \ref{t:eigenvector} remain valid.
We refer to \cite{MP11} for details and proofs. The modified localization procedures from \cite{HP12} and \cite{2013arXiv1312.5922P} with improved accuracy and stability properties might as well  be applied.

\subsection{Non-nested meshes and general coarsening}\label{ss:coarsening}
In Section~\ref{ss:mesh}, we have assumed that $\tri_h$ is derived from $\tri_H$ by some regular refinement, i.e., that the finite element meshes $\tri_h$ and $\tri_H$ are nested. This condition may be impracticable in relevant applications, e.g., in cases where the coefficient encodes microscopic geometric features such as jumps that require accurate resolution and the reasonable resolution can only be achieved by highly unstructured meshes (cf. Figure~\ref{fig:A3} in Section~\ref{ss:numexpunstructered} below).

A closer look to the previous error analysis shows that the nestedness of the underlying meshes is never used explicitly but enters only implicitly via the nestedness of corresponding spaces $V_H\subset V_h$. It turns out that all results generalize to the case where the standard finite element space $V_H$ on the coarse level is replaced with some general (possibly mesh free) coarse space $\tilde{V}_H\subset V_h$ with a local basis $\{\tilde{\phi}_j\}_{j\in J}$; $J$ being some finite index set. Precise necessary conditions for the theory read:
\begin{enumerate}
 \item[(a)] \emph{Local support and finite overlap.} For all $j\in J$, $\diam(\support \tilde{\phi}_j)\lesssim H$ and there is a finite number $\Col$ independent of $H$ such that no point $x\in\Omega$ belongs to the support of more than $\Col$ basis functions.
 \item[(b)] \emph{Non-negativity, continuity and boundedness.} For all $j\in J$, $\tilde{\phi}_j:\Omega\rightarrow [0,1]$ is continuous and $\|\nabla \tilde{\phi}_j\|_{L^{\infty(\Omega)}}\lesssim H^{-1}$.
 \item[(c)] \emph{Partition of unity up to a boundary strip.} For all $x\in\Omega$, it holds that $\dist(x,\partial\Omega)\lesssim H$ or $\sum_{j\in J}\tilde{phi}_j(x)=1$.
\end{enumerate}
Under the conditions (a)--(c), the operator $\IH$, defined by $\IH v:=\sum_{j\in J} \frac{(v,\tilde{\phi}_j)_{L^2(\Omega)}}{(1,\tilde{\phi}_j)_{L^2(\Omega)}}\tilde{\phi}_j$ for $v\in V$,  satisfies the required stability and approximation properties. Their proofs may easily be extracted from \cite{MR1706735}, where a slightly modified operator is considered. For details regarding the generalization of the decompositions and error bounds of this paper to some general coarse space characterized by (a)--(c), we refer to \cite{2013arXiv1312.5922P}, where everything (including the exponential decay of the coarse basis and its localization) has been worked out for a linear boundary value problem.

The conditions (a)--(c) are natural conditions for general coarse spaces used in domain decomposition methods and algebraic multigrid methods; see \cite[Ch. 3.10]{MR2104179} for an overview and \cite{MR1962682} for a particular construction without any coarse mesh. A very simple mesh-based construction which remains very close to the standard finite element space $V_H$ can be found in \cite[Section 2.2]{MR2861254} and works as follows. Given some regular fine mesh $\tri_h$, consider an arbitrary regular quasi-uniform coarse mesh $\tri_H$ with $H>h$. Let $V_h$ (resp. $V_H$) be the corresponding finite element space of continuous $\tri_h$-piecewise (resp. $\tri_H$-piecewise) affine functions and let $I^{\operatorname{nodal}}_h:V_H\subset C^0(\Omega) \rightarrow V_h$ denote the nodal interpolation operator with respect to the fine mesh. The nodal interpolation of standard nodal basis functions of the coarse mesh defines a nested initial coarse space
\begin{equation}\label{e:unstructured}
\tilde{V}_H:=\operatorname{span}\left\{I^{\operatorname{nodal}}_h\phi_z\;\vert\;z\in\N_H\right\}\subset V_h
\end{equation}
and $\Vc:=(1-\Pf)\tilde{V}_H$ is the corresponding coarse space of our method. The desired properties (a)--(c) of $\tilde{V}_H$ are proven in \cite[Lemma 2.1]{MR2861254}. Section~\ref{ss:numexpunstructered} shows numerical results based on this construction.

\subsection{Postprocessing}\label{ss:post}
\newcommand{\ucpost}{u_{\operatorname*{c,post}}}
As already mentioned in the introduction, the two-grid method of
\cite{MR1677419} allows a certain post-processing (solution of linear
problems on the fine scale) of coarse eigenpairs. So far, this method
was mainly used to post-process approximate eigenpairs of standard
finite element approximations on a coarse mesh, i.e., approximations
with respect to the space $V_H$. However, the framework presented in
\cite{MR1677419} is more general and readily applies to the modified
coarse space $\Vc$. Given some approximate eigenpair
$(\lambda_H^{(\ell)},\uc^{(\ell)})\in\mathbb{R}\times\Vc$ with
$\|\uc^{(\ell)}\|=1$ that solves \eqref{e:modeldisc}, the
post-processed approximate eigenfunction $\ucpost^{(\ell)}\in V_h$ is
characterized uniquely by
\begin{equation}\label{e:post}
  a(\ucpost^{(\ell)},v)=\lambda_H^{(\ell)}(\uc^{(\ell)},v)_{L^2(\Omega)}
\end{equation}
for all $v\in V_h$. The corresponding post-processed eigenvalue is 
\begin{equation}\label{e:post2}
  \lambda_{H,\operatorname{post}}^{(\ell)}:=\frac{a(\ucpost^{(\ell)},\ucpost^{(\ell)})}{(\ucpost^{(\ell)},\ucpost^{(\ell)})_{L^2(\Omega)}}.
\end{equation}
The error analysis of \cite{MR1677419} relies solely on the nestedness
$\Vc\subset V_h$ and, in essence, yields the error estimates
\begin{eqnarray*}
  \left|\lambda_{H,\operatorname{post}}^{(\ell)}-\lambda_h^{(\ell)}\right|&\leq&\tnorm{u_h^{(\ell)}-\ucpost^{(\ell)}}^2\\&\lesssim&\left
  (\lambda_H^{(\ell)}-\lambda_h^{(\ell)}\right)^2+\left(\lambda_h^{(\ell)}\right)^2\|u_h^{(\ell)}-\uc^{(\ell)}\|^2.
\end{eqnarray*}
The first estimate follows from \eqref{e:enl2} which remains valid for
$\uc^{(\ell)}$ and $\lambda_H^{(\ell)}$ replaced with $\ucpost^{(\ell)}$
and $\lambda_{H,\operatorname{post}}^{(\ell)}$. The second estimate
follows from the construction and standard inequalities (cf. \cite[Eq.
(4.3)]{MR1677419}).
Hence, with $u_h^{(\ell)}$ suitably chosen, Theorem~\ref{t:lower} and \ref{t:eigenvector} imply that
the error of the post-processed eigenvalues (resp. post-processed
eigenfunctions) is at least of order $H^6$ (resp. $H^4$). As for all our
previous results, the rates do not depend on any regularity of the
eigenfunctions.
In the third numerical experiment of Section~\ref{s:numexp} we will also
show results for this post-processing technique.

\subsection{Complexity}\label{ss:complex}
Finally, we shall comment on the overall complexity of our approach. Consider quasi-uniform meshes of size $H$ and $h$ and corresponding conforming first-order finite element spaces $V_H$ and $V_h$. We want to approximate the eigenvalues related to $V_h$.

In order to set up the coarse space $\Vc$, we need to solve $N_H$ linear problems with approximately $k^d N_h/N_H$ degrees of freedom each; the parameter $k$ being the truncation parameter as above. Since almost linear complexity is possible (using, e.g., multilevel preconditioning techniques), the cost for solving one of these problems up to a given accuracy is proportional to the number of degrees of freedom $N_h/N_H$ up to possible logarithmic factors. This yields an overall complexity of $k^d N_h\log(N_h)$ (resp. $N_H N_h\log(N_h)$ if $k^d\geq N_H$) for setting up the coarse problem. Note that this effort can be reduced drastically either by considering the independence of the linear problems in terms of parallelism or by exploiting a possible periodicity in the problem and the mesh.
In the latter case, only very few of the problems have to be computed because all the other ones are equivalent up to translation or rotation of coordinates.

On top of the assembling, an $N_H$-dimensional eigenvalue problem is to be solved. The complexity of this depends only on $N_H$, the number of eigenvalues of interest, and the truncation parameter $k$ but \emph{not} on the critically large parameter $N_h$.

The cost of the post-processing presented in Section~\ref{ss:post} is proportional to one fine solve for each eigenpair of interest, i.e., proportional to $N_h$ up to some logarithmic factor.

\section{Numerical Experiments}\label{s:numexp}
Three numerical experiments shall illustrate our theoretical results. While the first two experiments consider nested coarse and fine meshes, the third experiments uses the generalized coarsening strategy of Section~\ref{ss:coarsening}. In all experiments, we focus on the case without localization. The localization (as discussed in Section~\ref{ss:compression}) has been studied extensively for the linear problem in \cite{MP11,HP12,2013arXiv1308.3379H,2013arXiv1312.5922P} and for semi-linear problems in \cite{Henning.Mlqvist.Peterseim:2012}. In the present context of eigenvalue approximation, we are interested in observing the enormous convergence rate which is $4$ or higher for the eigenvalues. In order to achieve this rate also with truncation, patches have to be large (at least $4$ layers of elements) which pays off only asymptotically when $H$ is small enough.

\subsection{Constant coefficient on L-shaped domain}
Let $\Omega:=(-1,1)^2\setminus[0,1]^2$ be the L-shaped domain. Consider the constant scalar coefficient $A_1=1$ and uniform coarse meshes with mesh widths $\sqrt{2}H=2^{-1},\ldots,2^{-4}$ of $\Omega$ as depicted in Figure~\ref{f:tri}.
\begin{figure}
\begin{center}
\includegraphics[width=0.35\textwidth]{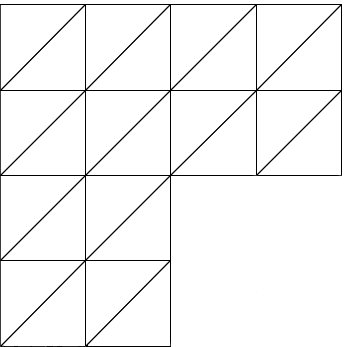}
\end{center}
\caption{Initial uniform triangulation of the $L$-shape domain (5 degrees of freedom).}\label{f:tri}
\end{figure}

The reference mesh $\tri_h$ has maximal mesh width $h=2^{-7}/\sqrt{2}$. We consider some $P1$ conforming finite element approximation of the eigenvalues on the reference mesh $\tri_h$ and compare these discrete eigenvalues $\lambda_h^{(\ell)}$ with coarse scale approximations depending on the coarse mesh size $H$.

Table~\ref{tab:errorA1} shows results for the case without truncation, i.e., all linear problems have been solved on the whole of $\Omega$.
\begin{table}[h]
\centering
\begin{tabular}{rrrrrr}
\hline
$\ell$&$\lambda_h^{(\ell)}$&$e^{(\ell)}(1/2\sqrt{2})$&$e^{(\ell)}(1/4\sqrt{2})$&$e^{(\ell)}(1/8\sqrt{2})$&$e^{(\ell)}(1/16\sqrt{2})$\\
\hline
1 &  9.6436568 & 0.004161918 & 0.000041786 & 0.000000696 & 0.000000014\\
2 & 15.1989733 & 0.009683715 & 0.000083718 & 0.000000888 & 0.000000011\\
3 & 19.7421815 & 0.024238729 & 0.000199984 & 0.000001930 & 0.000000022\\
4 & 29.5280022 & 0.084950011 & 0.000679046 & 0.000006309 & 0.000000074\\
5 & 31.9266947 & 0.120246865 & 0.001032557 & 0.000011298 & 0.000000169\\
6 & 41.4911125 & - & 0.002220585 & 0.000019622 & 0.000000264\\
7 & 44.9620831 & - & 0.002837949 & 0.000022540 & 0.000000257\\
8 & 49.3631818 & - & 0.003535358 & 0.000027368 & 0.000000295\\
9 & 49.3655616 & - & 0.004143842 & 0.000031434 & 0.000000343\\
10 & 56.7367306 & - & 0.006494922 & 0.000052862 & 0.000000606\\
11 & 65.4137240 & - & 0.013504833 & 0.000094150 & 0.000000995\\
12 & 71.0950435 & - & 0.013314963 & 0.000095197 & 0.000001077\\
13 & 71.6015951 & - & 0.011792861 & 0.000084001 & 0.000000851\\
14 & 79.0044010 & - & 0.021302527 & 0.000155038 & 0.000001526\\
15 & 89.3721008 & - & 0.038951872 & 0.000233603 & 0.000002613\\
16 & 92.3686575 & - & 0.042125029 & 0.000253278 & 0.000002442\\
17 & 97.4392146 & - & 0.033015921 & 0.000254700 & 0.000002435\\
18 & 98.7544790 & - & 0.039634464 & 0.000264156 & 0.000002482\\
19 & 98.7545515 & - & 0.046865242 & 0.000268012 & 0.000002500\\
20 & 101.6764284 & - & 0.045797998 & 0.000311683 & 0.000003071\\
\hline\\
\end{tabular}
\caption{Errors $e^{(\ell)}(H)=:\frac{\lambda_H^{(\ell)}-\lambda_h^{(\ell)}}{\lambda_h^{(\ell)}}$ for $\ell=1,\ldots,20$, constant coefficient $A_1$, and various choices of the coarse mesh size $H$.}
\label{tab:errorA1}
\end{table}
For fixed $\ell$, the rate of convergence of the eigenvalue error $\lambda_H^{(\ell)}-\lambda_h^{(\ell)}$ in terms of $H$ observed in Table~\ref{tab:errorA1} is between $6$ and $7$ which is even better than predicted in Theorem~\ref{t:lower} and in Remark~\ref{r:smooth}.

\subsection{Rough coefficient with multiscale features}
Let $\Omega:=(0,1)^2$ be the unit square. The scalar coefficient $A_2$ (see Figure~\ref{fig:A2}) is piecewise constant with respect to the uniform Cartesian grid of width $2^{-6}$. Its values are taken from the data of the SPE10 benchmark, see \texttt{http://www.spe.org/web/csp/}. The coefficient is highly varying and strongly heterogeneous. The contrast for $A_2$ is large, $\beta(A_2)/\alpha(A_1)\approx 4\cdot 10^6$. Consider uniform coarse meshes of size $\sqrt{2}H=2^{-1},2^{-2},\ldots,2^{-4}$ of $\Omega$ (cf. Figure~\ref{fig:A2}). Note that none of these meshes resolves the rough coefficient $A_2$ appropriately. Hence, (local) regularity cannot be exploited on coarse meshes.

Again, the reference mesh $\tri_h$ has width $h=2^{-7}/\sqrt{2}$ and we compare the discrete eigenvalues $\lambda_h^{(\ell)}$ (with respect to some $P1$ conforming finite element approximation of the eigenvalues on the reference mesh $\tri_h$) with coarse scale approximations depending on the coarse mesh size $H$. Table~\ref{tab:errorA2} shows the errors and allows us to estimate the average rate around $4$ which matches our expectation from the theory. We emphasize that the large contrast does not seem to affect the accuracy of our method in approximating the eigenvalues $\lambda_h^{(\ell)}$. However, the accuracy of $\lambda_h^{(\ell)}$ may be affected by the high contrast and the lack of regularity caused by the coefficient.

\begin{figure}
\begin{center}
\includegraphics[width=0.53\textwidth]{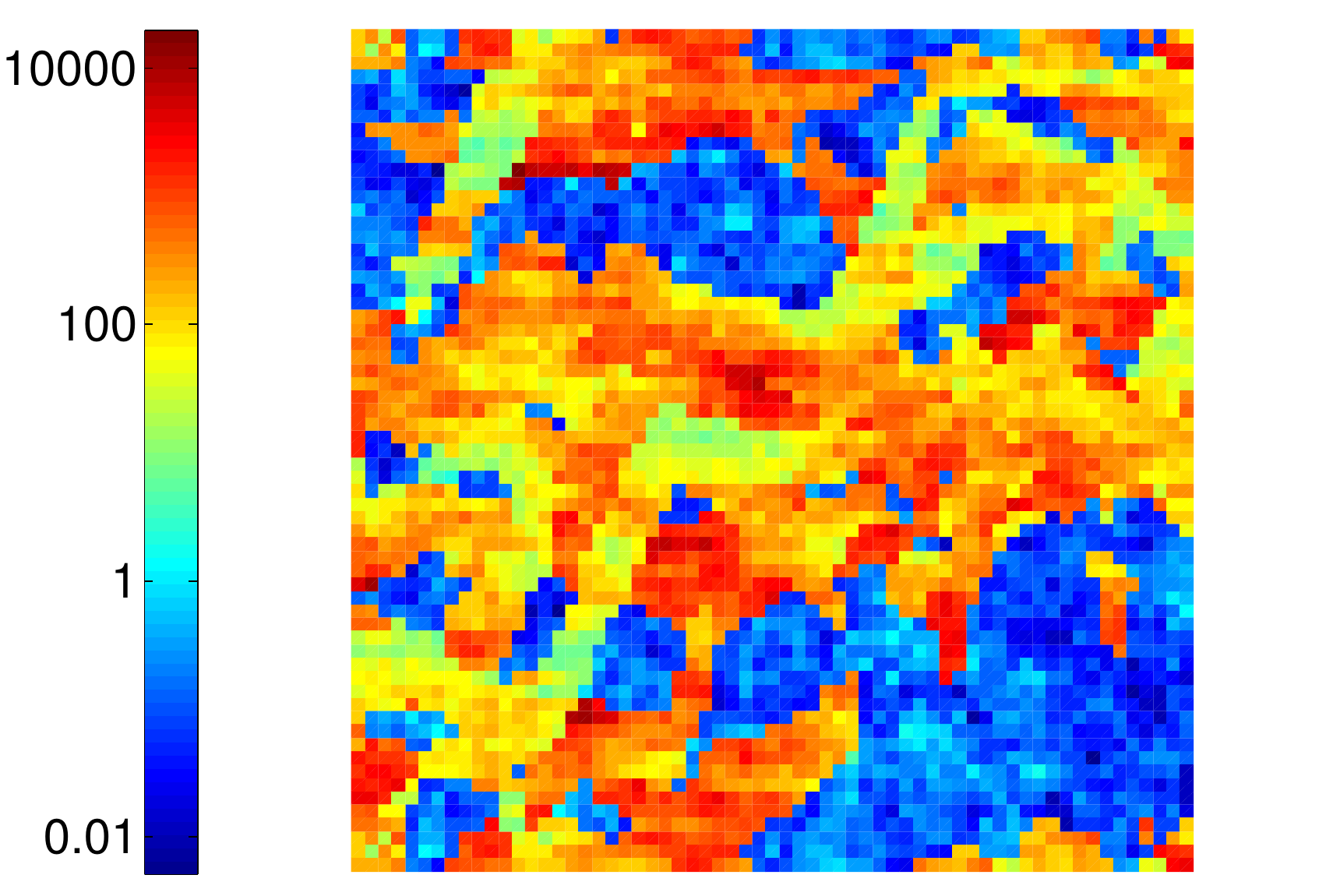}\includegraphics[width=0.35\textwidth]{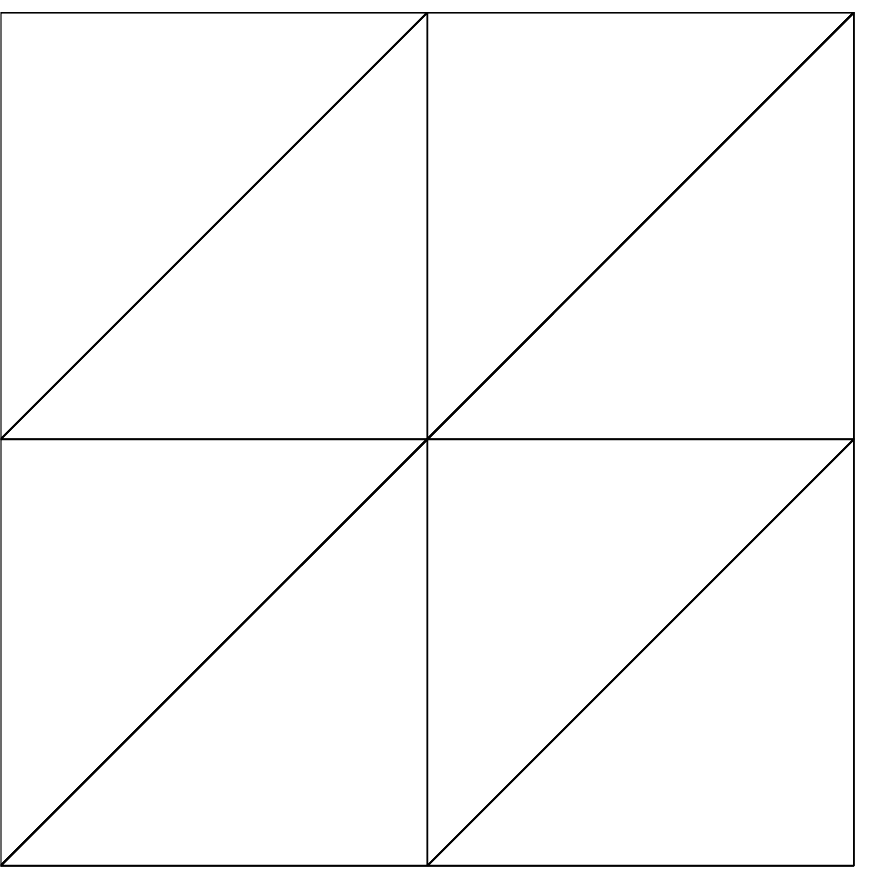}\end{center}
\caption{Scalar coefficient $A_2$ used in the second numerical experiment and initial uniform triangulation of the unit square (1 degree of freedom).}\label{fig:A2}
\end{figure}

\begin{table}[h]
\centering
\begin{tabular}{rrrrrr}
\hline
$\ell$&$\lambda_h^{(\ell)}$&$e^{(\ell)}(1/2\sqrt{2})$&$e^{(\ell)}(1/4\sqrt{2})$&$e^{(\ell)}(1/8\sqrt{2})$&$e^{(\ell)}(1/16\sqrt{2})$\\
\hline
 1&  21.4144522   &  5.472755371 &  0.237181706 &  0.010328293 &  0.000781683 \\
 2&  40.9134676 &                  - &  0.649080539 &  0.032761482 &  0.002447049 \\
 3&  44.1561133 &                  - &  1.687388874 &  0.097540102 &  0.004131422 \\
 4&  60.8278691 &                  - &  1.648439518 &  0.028076168 &  0.002079812\\
 5&  65.6962136 &                  - &  2.071005692 &  0.247424446 &  0.006569640 \\
 6&  70.1273082 &                  - &  4.265936007 &  0.232458016 &  0.016551520 \\
 7&  82.2960238 &                  - &  3.632888104 &  0.355050163 &  0.013987920 \\
 8&  92.8677605 &                  - &  6.850048057 &  0.377881216 &  0.049841235 \\
 9&  99.6061234 &                  - &  10.305084010 &  0.469770376 &  0.026027378 \\
 10&  109.1543283 &                  - &                  - &  0.476741452 &  0.005606426 \\
 11&  129.3741945 &                  - &                  - &  0.505888044 &  0.062382302 \\
 12&  138.2164330 &                  - &                  - &  0.554736550 &  0.039487317 \\
 13&  141.5464639 &                  - &                  - &  0.540480876 &  0.043935515 \\
 14&  145.7469718 &                  - &                  - &  0.765411709 &  0.034249528 \\
 15&  152.6283573 &                  - &                  - &  0.712383825 &  0.024716759 \\
 16&  155.2965039 &                  - &                  - &  0.761104705 &  0.026228034 \\
 17&  158.2610708 &                  - &                  - &  0.749058367 &  0.091826207 \\
 18&  164.1452194 &                  - &                  - &  0.840736127 &  0.118353184 \\
 19&  171.1756923 &                  - &                  - &  0.946719951 &  0.111314058 \\
 20&  179.3917590 &                  - &                  - &  0.928617606 &  0.119627862 \\
\hline\\
\end{tabular}
\caption{Errors $e^{(\ell)}(H)=:\frac{\lambda_H^{(\ell)}-\lambda_h^{(\ell)}}{\lambda_h^{(\ell)}}$ for $\ell=1,\ldots,20$, rough coefficient $A_2$, and various choices of the coarse mesh size $H$.}
\label{tab:errorA2}
\end{table}

\subsection{Particle composite modeled by an unstructured mesh}\label{ss:numexpunstructered}
\begin{figure}
\begin{center}
\includegraphics[width=0.48\textwidth]{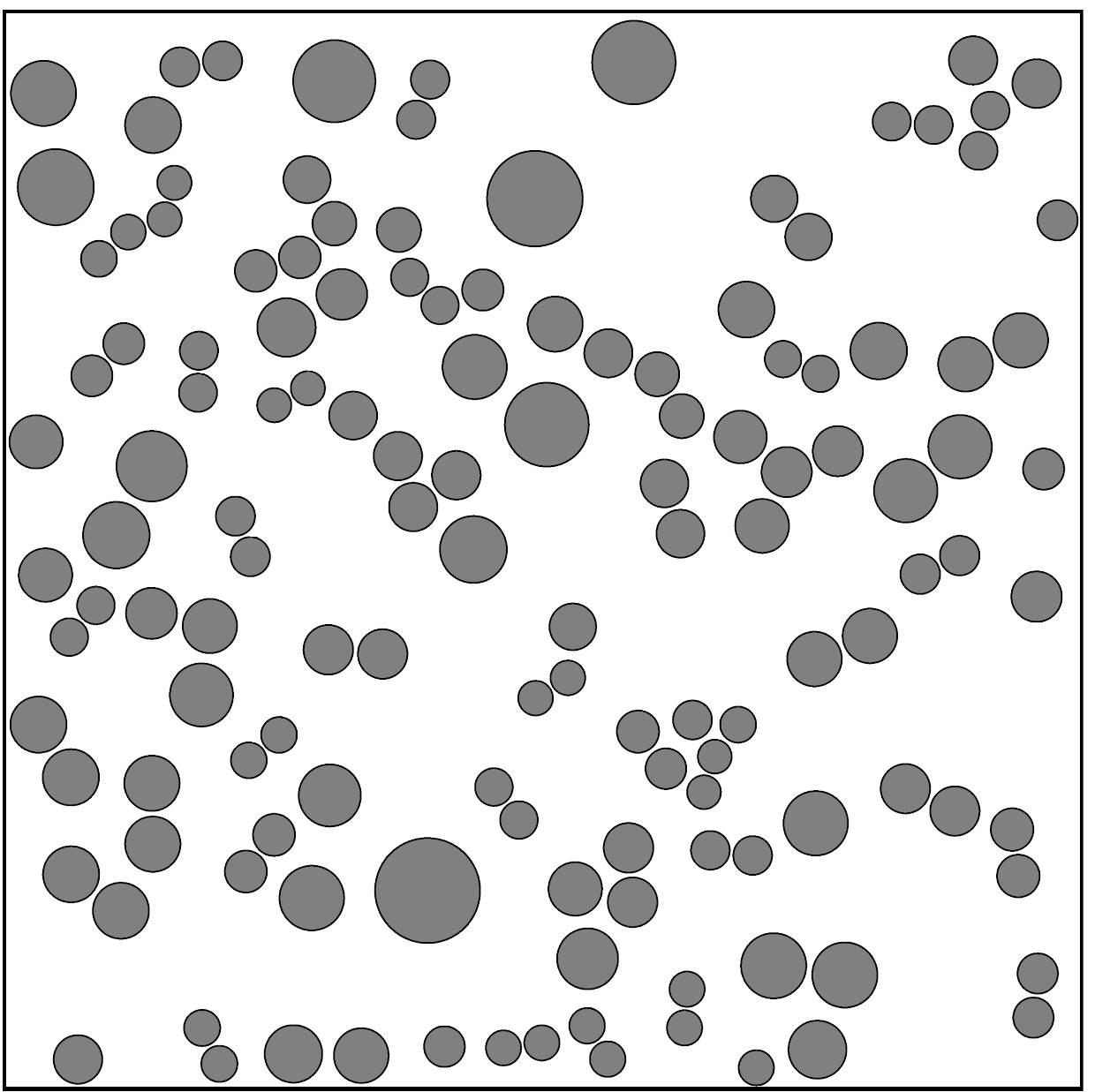}\hspace{0.01\textwidth}\includegraphics[width=0.48\textwidth]{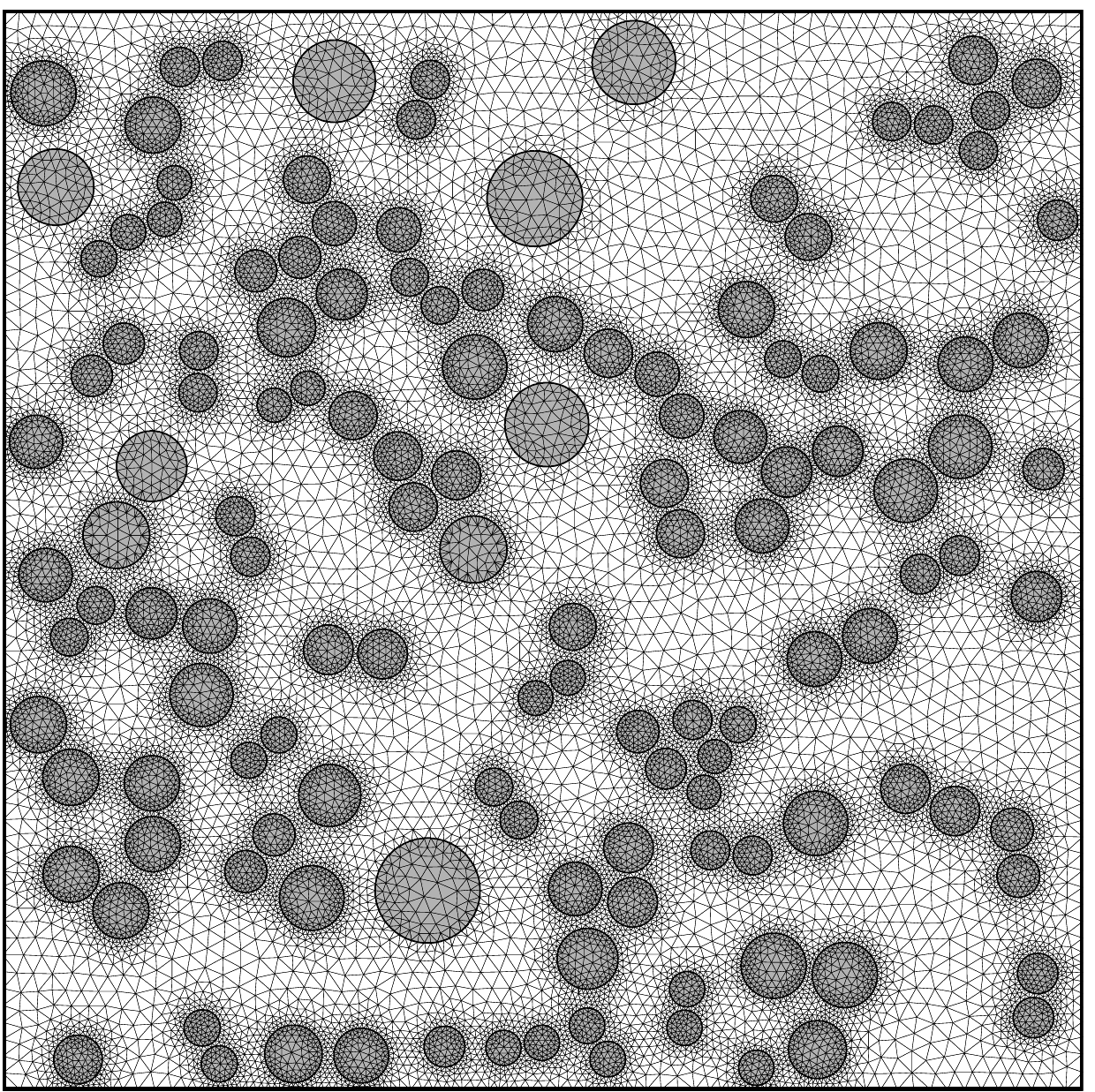}\end{center}
\caption{Left: Scalar coefficient $A_3$ used in the third numerical experiment. $A_3$ takes the value $100$ in the gray shaded inclusions and the value $1$ elsewhere. Right: Unstructered fine mesh $\tri_h$ aligned with jumps of the coefficient $A_3$.}\label{fig:A3}
\end{figure}
Let $\Omega:=(0,1)^2$ be the unit square. In this experiment, the scalar coefficient $A_3$ models heat conductivity in some model composite material with randomly dispersed circular inclusions as depicted in Figure~\ref{fig:A3}. The coefficient $A_3$ takes the value $100$ in the gray shaded inclusions and the value $1$ elsewhere. In order to resolve the discontinuities, we simply align the fine mesh $\tri_h$ with the boundaries of the inclusions (see Figure~\ref{fig:A3}). The mesh size of $\tri_h$ satisfies $2^{-9}\lesssim h\lesssim 2^{-7}$. Note that this fine mesh $\tri_h$ is solely based on geometric resolution and shape regularity. The grading towards the inclusions is not adapted to the characteristic behavior of the eigenfunctions. However, this mesh might be the actual output of some commercial mesh generator or modeling tool. Sufficient resolution could be achieved with fewer degrees of freedom, however, this would require more sophisticated discretization spaces; we refer to \cite{MR2684351,Peterseim:2013,Peterseim.Carstensen:2013} for possible choices and further references.

As in the previous experiment, we consider uniform coarse meshes of size $\sqrt{2}H=2^{-1},2^{-2},\ldots,2^{-4}$ of $\Omega$ (cf. Figure~\ref{fig:A2}). Note that these meshes neither resolves the coefficient $A_3$ appropriately nor can be refined to $\tri_h$ in a nested way.
For the construction of the upscaling approximation we employ the generalized coarse space defined in \eqref{e:unstructured} in Section~\ref{ss:coarsening}. We compare the discrete eigenvalues $\lambda_h^{(\ell)}$ (with respect to some $P1$ conforming finite element approximation of the eigenvalues on the reference mesh $\tri_h$) with coarse scale approximations depending on the coarse discretization parameter $H$. Table~\ref{tab:errorA3} shows the results which clearly support our claim that the nestedness of coarse and fine meshes is not essential and that upscaling far beyond the characteristic length scales of the problem (i.e., the radii of the inclusions and their distances) is possible.

For this problem, we have also computed the post-processed approximations according to Section~\ref{ss:post}. Table~\ref{tab:errorA32} shows the error for the eigenvalues which are more accurate by several orders of magnitude. The experimental rates are roughly between $5$ and $6$ in Table~\ref{tab:errorA3} without post-processing and around $9$ to $10$ after post-processing in Table~\ref{tab:errorA32}.

\begin{table}[h]
\centering
\begin{tabular}{rrrrrr}
\hline
$\ell$&$\lambda_h^{(\ell)}$&$e^{(\ell)}(1/2\sqrt{2})$&$e^{(\ell)}(1/4\sqrt{2})$&$e^{(\ell)}(1/8\sqrt{2})$&$e^{(\ell)}(1/16\sqrt{2})$\\
\hline
1 & 25.6109462 & 0.025518831 & 0.000572341 & 0.000017083 & 0.000000700\\
2 & 58.9623566 & - & 0.005235813 & 0.000090490 & 0.000002710\\
3 & 67.5344854 & - & 0.006997582 & 0.000154850 & 0.000006488\\
4 & 98.2808694 & - & 0.023497502 & 0.000358178 & 0.000011675\\
5 & 121.2290664 & - & 0.052366141 & 0.000563438 & 0.000016994\\
6 & 125.2014779 & - & 0.066627585 & 0.000747688 & 0.000019934\\
7 & 156.0597873 & - & 0.145676350 & 0.001579177 & 0.000034329\\
8 & 168.2376096 & - & 0.095360287 & 0.001320185 & 0.000043781\\
9 & 197.4467434 & - & 0.343991317 & 0.002888471 & 0.000049479\\
10 & 209.4657306 & - & - & 0.003223901 & 0.000056318\\
11 & 222.4472476 & - & - & 0.003431462 & 0.000080284\\
12 & 245.5656759 & - & - & 0.005906282 & 0.000102243\\
13 & 253.7074603 & - & - & 0.006215809 & 0.000121646\\
14 & 288.0756442 & - & - & 0.013859535 & 0.000180899\\
15 & 298.8903269 & - & - & 0.010587124 & 0.000138404\\
16 & 311.4410556 & - & - & 0.012159268 & 0.000161510\\
17 & 324.6865434 & - & - & 0.012143676 & 0.000176624\\
18 & 336.7931865 & - & - & 0.016554437 & 0.000233067\\
19 & 379.5697606 & - & - & 0.023254268 & 0.000325324\\
20 & 386.9938901 & - & - & 0.028772395 & 0.000383532\\
\hline\\
\end{tabular}
\caption{Errors $e^{(\ell)}(H)=:\frac{\lambda_H^{(\ell)}-\lambda_h^{(\ell)}}{\lambda_h^{(\ell)}}$ for $\ell=1,\ldots,20$, coefficient $A_3$, and various choices of the coarse mesh size $H$.}
\label{tab:errorA3}
\end{table}

\begin{table}[h]
\centering
\begin{tabular}{rrrrrr}
\hline
$\ell$&$\lambda_h^{(\ell)}$&$e^{(\ell)}(1/2\sqrt{2})$&$e^{(\ell)}(1/4\sqrt{2})$&$e^{(\ell)}(1/8\sqrt{2})$&$e^{(\ell)}(1/16\sqrt{2})$\\
\hline
1 & 25.6109462 & 0.001559704 & 0.000003765 & 0.000000008 & 3.5e-10 \\
2 & 58.9623566 & - & 0.000191532 & 0.000000213 & 1.9e-08 \\
3 & 67.5344854 & - & 0.000284980 & 0.000000474 & 0.000000001 \\
4 & 98.2808694 & - & 0.002239689 & 0.000002253 & 0.000000004 \\
5 & 121.2290664 & - & 0.007461217 & 0.000005065 & 0.000000008 \\
6 & 125.2014779 & - & 0.011284614 & 0.000006826 & 0.000000008 \\
7 & 156.0597873 & - & 0.042466017 & 0.000023867 & 0.000000024 \\
8 & 168.2376096 & - & 0.025093182 & 0.000027547 & 0.000000042 \\
9 & 197.4467434 & - & 0.186960343 & 0.000072471 & 0.000000051 \\
10 & 209.4657306 & - & - & 0.000105777 & 0.000000079 \\
11 & 222.4472476 & - & - & 0.000131569 & 0.000000129 \\
12 & 245.5656759 & - & - & 0.000286351 & 0.000000213 \\
13 & 253.7074603 & - & - & 0.000268463 & 0.000000255 \\
14 & 288.0756442 & - & - & 0.000915102 & 0.000000473 \\
15 & 298.8903269 & - & - & 0.000762135 & 0.000000403 \\
16 & 311.4410556 & - & - & 0.000873769 & 0.000000504 \\
17 & 324.6865434 & - & - & 0.000955392 & 0.000000642 \\
18 & 336.7931865 & - & - & 0.001335246 & 0.000000977 \\
19 & 379.5697606 & - & - & 0.002896202 & 0.000001886 \\
20 & 386.9938901 & - & - & 0.007202657 & 0.000001908 \\
\hline\\
\end{tabular}
\caption{Errors $e^{(\ell)}(H)=:\frac{\lambda_{H,\operatorname{post}}^{(\ell)}-\lambda_h^{(\ell)}}{\lambda_h^{(\ell)}}$ after post-processing for $\ell=1,\ldots,20$, coefficient $A_3$, and various choices of the coarse mesh size $H$.}
\label{tab:errorA32}
\end{table}

\bibliographystyle{alpha}

\end{document}